\documentclass[10pt,a4paper]{article}
\usepackage[utf8]{inputenc}
\usepackage[english]{babel}
\usepackage{multicol,graphicx,color}
\usepackage{pslatex}
\usepackage{authblk}
\usepackage{amsthm}
\usepackage{amsmath}
\usepackage{amssymb}
\usepackage{latexsym}
\usepackage{lscape}
\usepackage{epsfig}
\usepackage{amsfonts}
\usepackage{mathrsfs,bbm}
\usepackage[
hmarginratio={1:1},     
vmarginratio={1:1},     
textwidth=15cm,        
textheight=21cm,
heightrounded,          
]{geometry}

\usepackage{tikz}

\makeatletter

\theoremstyle{plain}

\numberwithin{equation}{section}
\newtheorem{theorem}{Theorem}[section]
\newtheorem{lemma}[theorem]{Lemma}
\newtheorem{proposition}[theorem]{Proposition}

\theoremstyle{definition}
\newtheorem{defn}{Definition}[section]

\theoremstyle{remark}
\newtheorem{remk}{Remark}[section]

\begin{document}

\title{{\Large\bf{Normalized solutions for Choquard equations with critical nonlinearities on bounded domains}}}

\author{Ru Yan$^{\mathrm{a,b}}$ \\
{\small $^{\mathrm{a}}$Institute of Mathematics, AMSS, Chinese Academy of Science,
Beijing 100190, China}\\
{\small $^{\mathrm{b}}$University of Chinese Academy of Science,
Beijing 100049, China}\\
}

\date{}
\maketitle

\begin{abstract}
The aim of this work is the study of the existence of normalized solutions to the nonlinear Schrödinger equation with nonlocal nonlinearities:
    \begin{equation}\nonumber
        \left\{\begin{aligned}
            &-\Delta u =\lambda u+(I_\alpha*|u|^{2_\alpha^*})|u|^{2_\alpha^*-2}u+a(I_\alpha*|u|^p)|u|^{p-2}u,\ x\in\Omega,\\
            &u>0\ \text {in}\ \Omega,\ u=0\ \text {on}\ \partial \Omega,\  \int _{\Omega}|u|^2dx=c,\\
        \end{aligned}
        \right.
    \end{equation}
    where $c>0,\ \alpha \in (0,N),\ \frac{N+\alpha+2}{N}<p<\frac{N+\alpha}{N-2}=2_\alpha^*,\ a\ge 0,\ \Omega \subset \mathbb{R}^N (N \ge 3)$ is smooth, bounded, star-shaped and $I_\alpha$ is the Riesz potential. We prove the existence of two positive normalized solutions, one of which is a ground state and the other is a mountain pass solution.\\
\end{abstract}

\medskip

{\small \noindent \text{Key Words:} Normalized solutions, Ground state, Choquard equation, Brezis-Nirenberg problem.
	
\medskip

{\small \noindent \text{Data availability:} Data sharing is not applicable to this article as no datasets were generated or analysed during the current study.}

\section{Introduction}
We study the nonlocal nonlinear equation 
\begin{equation}\label{aa}
    \left\{\begin{aligned}
        &-\Delta u =\lambda u+(I_\alpha*|u|^{2_\alpha^*})|u|^{2_\alpha^*-2}u+a(I_\alpha*|u|^p)|u|^{p-2}u,\ x\in\Omega,\\
        &u>0\ \text {in}\ \Omega,\ u=0\ \text {on}\ \partial \Omega,\  \int _{\Omega}|u|^2dx=c,\\
    \end{aligned}
    \right.
\end{equation}
where $\ c>0,\ \alpha \in (0,N),\ \frac{N+\alpha+2}{N}=2_\alpha^\sharp<p<2_\alpha^*=\frac{N+\alpha}{N-2},\ a\ge0,\ \Omega \subset \mathbb{R}^N(N\ge 3)$ is smooth, bounded, star-shaped and $ I_\alpha:\mathbb{R}^N \to \mathbb{R}$ is the Riesz potential defined by
$$
I_{\alpha}(x)=\frac{\mathcal{A}(N,\alpha )}{|x|^{N-\alpha}} \text { with } \mathcal{A}(N,\alpha )=\frac{\Gamma\left(\frac{N-\alpha}{2}\right)}{\Gamma\left(\frac{\alpha}{2}\right) \pi^{N / 2} 2^{\alpha}} \text { for } x \in \mathbb{R}^{N} \backslash\{0\}.
$$

Equation \eqref{aa} originates from physics. For $N=3,\ \alpha =2,\ a=0$, it covers in particular the Choquard-Pekar equation, which traces back to the work conducetd by S.I. Pekar \cite{Pekar+1954} in 1954 who describes the quantum mechanics of a polaron at rest(see discussion in \cite{MR1462224}). The Choquard equation is also introduced by P. Choquard \cite{MR471785} in 1976 to describe an electron trapped in its own hole, in a certain approximation to Hartree-Fock theory of one component plasma. Additionally, it also appears in dealing with the Schrödinger equation of quantum physics together
with nonrelativistic Newtonian gravity \cite{MR1386305}, which is the reason why it is known as the Schrödinger-Newton equation. The solutions of equation \eqref{aa} correspond to standing waves of the focusing time-dependent Hartree equation
\begin{equation}\label{ab}
    \left\{\begin{aligned}
        &i\partial_t \psi +\Delta \psi+(I_\alpha*|\psi|^{2_\alpha^*})|\psi|^{2_\alpha^*-2}\psi+a(I_\alpha*|\psi|^p)|\psi|^{p-2}\psi=0,\ (t,x)\in\mathbb{R}\times \Omega,\\
        &\psi(t,x)=0,(t,x)\in\mathbb{R}\times \partial \Omega.\\
    \end{aligned}
    \right.
\end{equation}
If $u$ solves \eqref{aa}, then the form $\psi(t,x)=e^{i\lambda t}u(x)$ is the solution to \eqref{ab}. 

To the best of our knowledge, there are two different approaches to study the existence of solutions of equation (1.1) using variational methods. One way is to study the fixed frequency problem, and we recommend that readers consult the survey paper \cite{MR3625092} and its comprehensive list of references for further details. The other way is to study the fixed $L^2$-norm which has attracted considerable interest from physicists. That is, we search for critical points of $E:H^1(\Omega)\to\mathbb{R}$ given by
\begin{align*}
    E(u)=\frac{1}{2}\int_{\Omega}|\nabla u|^2 d x-\frac{1}{2\cdot 2_\alpha^*}  {\mathbb{D}}(u^+)-\frac{a}{2p}{\mathbb{B}_p}(u^+),
\end{align*}
submitted to the constraint $S_c$, where 
\begin{align*}
    &S_c:=\{u \in H_0^1(\Omega) \text { s.t.} \int_{\Omega}\left|u\right|^2 d x=c\},\\
    &{\mathbb{D}}(u):=\int_{\Omega }(I_\alpha*|u|^{2_\alpha^*})|u|^{2_\alpha^*}dx,\\
    &{\mathbb{B}_p}(u):=\int_{\Omega }(I_\alpha*|u|^{p})|u|^{p}dx.
\end{align*}
The mass $\left \| \Psi(t,\cdot ) \right \|_{L^2}= \left \| u \right \|_{L^2}$ represents the power supply in the nonlinear optics framework, or the number of particles of each component in Bose-Einstein condensates. Moreover, the existence of normalized solution is key to study orbital stability for standing waves.  

Let us recall some results of normalized solution and describe our research motivation. Consider the stationary NLS equation
\begin{equation}\label{ac}
    \left\{\begin{aligned}
        &-\Delta u =\lambda u+f(u),\ x\in\Omega,\\
        &\int _{\Omega}|u|^2dx=c,u>0.\\
    \end{aligned}
    \right.
\end{equation}
Define 
$$
J(u)=\frac{1}{2}\int_{\Omega}|\nabla u|^2 d x-\frac{1}{p}\int_{\Omega}F(u) dx,
$$
where $F(u)$ is the primitive function of $f(u)$. When $p\in(2+\frac{4}{N},2^*)$, the functional $J$ is unbounded, and we can't obtain the existence of a global minimizer by minimizing $J$ on the constraint $S_c$. When $\Omega=\mathbb{R}^N$, Jeanjean \cite{MR1430506} obtains normalized solution, and the dilation $t\in(0,+\infty)\to t*u=t^{\frac{2}{N}}u(tx)$ is the key to constructing a bounded Palai-Smale sequence. Since then, normalized solutions on entire space have been widely studied in the mass supercritical case. However, there are only a few works \cite{MR4603876,MR3689156,MR3918087,MR3318740,song2024,MR4191345,MR4803490,chang2025,MR4845016} dealing with \eqref{ac} in bounded domains, since the dilation $t*u$ is not available here. It seems that Noris, Tavares and Verzini \cite{MR3318740} study \eqref{ac} on the unitary ball with $f(u)=|u|^{p-2}u$ and $p$ is Sobolev-subcritical. And they obtain a positive solution which is a local minimizer in general bounded domains with the Sobolev critical case \cite{MR3918087}. Subsequently, Song and Zou \cite{song2024}, as well as Pierotti, Verzini and Yu \cite{MR4847285} take star-shaped domains into consideration, and identify a second positive normalized solution at the mountain pass level with small mass. Recently, Chang, Liu and Yan \cite{chang2025} establish the existence of a second positive solution which is of M-P type in general bounded domains. 

To the best of our knowledge, no one has considered normalized solutions on bounded domains to Choquard equation. A natural question which is open is:
\begin{itemize}
	\item[(Q1)] Whether \eqref{aa} has normalized solutions in bounded domains? 
\end{itemize}   
 
In the paper, we give a positive answer to this open question for shaped domains following the framework of \cite{song2024}. Denote 
\begin{align*}
    E(u)=\frac{1}{2}\int_{\Omega}|\nabla u|^2 d x-\frac{1}{2\cdot 2_\alpha^*}{\mathbb{D}}(u^+)-\frac{a}{2p}{\mathbb{B}_p}(u^+),
\end{align*}
where $u^+=\max \left \{ u,0 \right \}.$
In order to find positive solution, we search for critical points of $E(u)$ under constraint $S_c^+$, where 
$$
S_c^+:=\{u \in H_0^1(\Omega) \text { s.t.} \int_{\Omega}\left|u^+\right|^2 d x=c\}.\ 
$$  

Recall Hardy-Littlewood-Sobolev's inequality which can be found in \cite{0Analysis}.
\begin{lemma}\label{A}
    Let $N\ge1$ and $0<\alpha<N$ with $1/r+(N-\alpha)/ N+1/s=2, f \in L^r\left(\mathbb{R}^N\right)$ and $h \in L^s\left(\mathbb{R}^N\right)$. There exists a sharp constant $\mathcal{C}_H(N, \alpha, r) $, independent of $f$ and $h$, such that
\begin{align}\label{ae}
    \int_{\mathbb{R}^N} \int_{\mathbb{R}^N} \frac{f(x) g(y)}{|x-y|^{N-\alpha}} d x d y \leqslant \mathcal{C}_H(N, \alpha, r)|f|_r|h|_s.
\end{align}
If $r=s=2 N /(N+\alpha)$, then
$$
\mathcal{C}_H(N, \alpha)=\mathcal{C}_H(N,\alpha,\frac{2N}{N+\alpha})=\pi^{\frac{N-\alpha}{2}} \frac{\Gamma\left(\frac{\alpha}{2}\right)}{\Gamma\left(\frac{N+\alpha}{2}\right)}\left\{\frac{\Gamma\left(\frac{N}{2}\right)}{\Gamma(N)}\right\}^{-\frac{\alpha}{N}}.
$$
In this case, the equality in \eqref{ae} holds if and only if $f \equiv \mathcal{C}_H(N,\alpha)h$ and
$$
h(x)=C\left(\gamma^2+|x-a|^2\right)^{-(N+\alpha) / 2}
$$
for some $C \in \mathbb{C}, 0 \neq \gamma \in \mathbb{R}$ and $a \in \mathbb{R}^N$.
\end{lemma}

From the Hardy-Littlewood-Sobolev inequality, for all $u \in D^{1,2}\left(\mathbb{R}^N\right)$ we know
$$
\int_{\mathbb{R}^N} \int_{\mathbb{R}^N} \frac{|u(x)|^{2_\alpha^*}|u(y)|^{2_\alpha^*}}{|x-y|^{N-\alpha} } d x d y\le \mathcal{C}_H(N,\alpha)|u|_{2^*}^{22_\alpha^*}\le \mathcal{C}_H(N,\alpha)\mathcal{S}^{-2_\alpha^*}\left ( \int_{\mathbb{R}^N}|\nabla u|^2 d x \right )^{2_\alpha^*},
$$
where $\mathcal{C}_H(N,\alpha)$ is defined as in Lemma \ref{A} and $\mathcal{S}$ is the best Sobolev constant. We define
\begin{align*}
    \mathcal{S}_H:=\inf _{u \in D^{1,2}\left(\mathbb{R}^N\right) \backslash\{0\}}\frac{\int_{\mathbb{R}^N}|\nabla u|^2 d x}{\left({\mathcal{A}(N,\alpha )}\int_{\mathbb{R}^N} \int_{\mathbb{R}^N} \frac{|u(x)|^{2_\alpha^*}|u(y)|^{2_\alpha^*}}{|x-y|^{N-\alpha} } d x d y\right)^{\frac{1}{2_\alpha^*}}},
\end{align*}
which is reached (see \cite{MR4027015}) and 
\begin{align*}
    \mathcal{S}_H=\frac{\mathcal{S}}{[\mathcal{A}(N,\alpha )\mathcal{C}_H(N,\alpha)]^{\frac{N-2}{N+\alpha}}}.
\end{align*}
The Gagliardo-Nirenberg inequality of Hartree type \cite{MR3056699} also is an important tool in the paper:
\begin{align}\label{ae0}
    \int_{\Omega }(I_\alpha*|u|^{p})|u|^{p}dx\le \mathcal{C}_G\left \|\nabla u\right \|_2^{2p\eta_p}\left \| u\right \|_2^{2p(1-\eta_p)}, 
\end{align}
where $\eta_p=\frac{Np-N-\alpha }{2p},\ p\in[2_\alpha,2_\alpha^*],\ N\ge 1,\ 0<\alpha <N$ and $\mathcal{C}_G=\mathcal{C}_G(N,\alpha ,p)$ is the best constant.

Next, we will introduce a set, which plays a key role in this paper. As we all know, any critical point $u$ of $\left.E\right|_{S_c^+}$ satisfies the following Pohozaev identity:
\begin{align}
    \frac{N-2}{2} \int_{\Omega}|\nabla u|^2 d x+\frac{1}{2}\int_{\partial \Omega} \left|\nabla u\right|^2\sigma\cdot n d\sigma-\frac{N}{2} \lambda\|u\|_{2}^{2}-\frac{N-2}{2}{\mathbb{D}}(u^+)-\frac{N+\alpha}{2p} a {\mathbb{B}_p}(u^+)=0.\label{ae1}
\end{align}
Furthermore, noticing $u$ is a solution of \eqref{aa}, we also have:
$$
\int_{\Omega}|\nabla u|^2 d x-\lambda |u|_2^2-{\mathbb{D}}(u^+)-a{\mathbb{B}_p}(u^+)=0. 
$$
Note that $\sigma\cdot n > 0$ since $\Omega$ is star-shaped with respect to the origin. Combining above two equations, we get any critical point $u \in \mathcal{G}$, where 
$$
\mathcal{G}:=\left\{u \in S_c^+: \int_{\Omega}\left|\nabla u\right|^2 dx>{\mathbb{D}}(u^+)+a\eta_p{\mathbb{B}_p}(u^+)\right\}\\
$$
with boundary
$$
\partial \mathcal{G}:=\left\{u \in S_c^+: \int_{\Omega}\left|\nabla u\right|^2 dx={\mathbb{D}}(u^+)+a\eta_p{\mathbb{B}_p}(u^+)\right\}.
$$
It is easy to see that $E$ is bounded from below on $\mathcal{G}$ which implies the boundedness of any sequence contained in $\mathcal{G}$ with bounded energy. Additionally, $\int_{\Omega}\left|\nabla u\right|^2 d x$ has a positive lower bound on $\partial \mathcal{G}$. Indeed, For $a= 0$, 
\begin{align}
    E(u)&>\left (\frac{1}{2}-\frac{1}{2\cdot2_\alpha^*} \right )\int_{\Omega}\left|\nabla u\right|^2 dx \ge \frac{\alpha+2}{2N+2\alpha}\int_{\Omega}\left|\nabla u\right|^2 dx,\ \forall u\in \mathcal{G}, \label{af} 
\end{align}
\begin{align}
    \int_{\Omega}\left|\nabla u\right|^2 dx={\mathbb{D}}(u^+)\le \mathcal{S}_H^{-2_\alpha^*}\left (\int_{\Omega}\left|\nabla u\right|^2 dx\right )^{2_\alpha^*},\ \forall u\in \partial\mathcal{G}. \label{ag} 
\end{align}
For $a>0$, $2_\alpha^\sharp<p<2_\alpha^*$,
\begin{align}
    E(u)=\frac{1}{2}\int_{\Omega}\left|\nabla u\right|^2 dx-\frac{1}{2p\eta_p}(\frac{p\eta_p}{2_\alpha^*}{\mathbb{D}}(u^+)+a\eta_p{\mathbb{B}_p}(u^+)) \nonumber \\ 
    \ge \left ( \frac{1}{2}-\frac{1}{2p\eta_p}\right ) \int_{\Omega}\left|\nabla u\right|^2 dx,\ \forall u\in \mathcal{G}. \label{ah} 
\end{align}
Using the Gagliardo-Nirenberg inequality of Hartree type \eqref{ae0}, we also have
\begin{align}
    \int_{\Omega}\left|\nabla u\right|^2 dx =&{\mathbb{D}}(u^+)+a\eta_p{\mathbb{B}_p}(u^+)
    \le\mathcal{S}_H^{-2_\alpha^*}\left (\int_{\Omega}\left|\nabla u\right|^2 dx\right )^{2_\alpha^*} \nonumber\\
    &+a\eta_p\mathcal{C}_G(N,\alpha ,p)c^{p(1-\eta_p)}\left (\int_{\Omega}\left|\nabla u\right|^2 dx\right )^{p\eta_p},\ \forall u\in \partial\mathcal{G}.\label{ai}
\end{align}

\begin{remk}
    When $a>0$ and $2_\alpha< p \le 2_\alpha^\sharp$, we get
    \begin{align*}
        E(u)&> \frac{\alpha+2}{2N+2\alpha}\int_{\Omega}\left|\nabla u\right|^2 dx-a(\frac{1}{2p}-\frac{\eta_p}{2\cdot 2_\alpha^*}){\mathbb{B}_p}(u^+)\\
        &\ge \frac{\alpha+2}{2N+2\alpha}\int_{\Omega}\left|\nabla u\right|^2 dx -a(\frac{1}{2p}-\frac{\eta_p}{2\cdot 2_\alpha^*})\mathcal{C}_G^pc^{p(1-\eta_p)}(\int_{\Omega}\left|\nabla u\right|^2 dx)^{p\eta_p/2}, \ \forall u\in \mathcal{G}. 
    \end{align*}
    The argument of the above case is technical, so it is not included in this paper. 
\end{remk}
Based on the observations above, we obtain a normalized ground state to \eqref{aa} in the following sense.
\begin{defn}
    Let 
    $$
    S_c:=\{u \in H_0^1(\Omega) \text { s.t. } \int_{\Omega}\left|u\right|^2 d x=c\}.
    $$
    We say $u\in S_c$ is a normalized ground state to \eqref{aa} if  
    $$
    \widetilde{E}^{\prime} |_{S_c}\left(u_c\right)=0 \text{ and }
    \widetilde{E} (u_c)=\inf \left \{ \widetilde{E} (u):u\in S_c,\ \widetilde{E}^{\prime} |_{S_c}\left(u\right)=0\right \}, 
    $$
    where
    \begin{align*}
        \widetilde{E}(u)=\frac{1}{2}\int_{\Omega}|\nabla u|^2 d x-\frac{1}{2\cdot 2_\alpha^*}{\mathbb{D}}(u)-\frac{a}{2p}{\mathbb{B}_p}(u).
    \end{align*}
\end{defn}

We have
\begin{theorem}\label{C}
    Let $\Omega$ be smooth, bounded, and star-shaped with respect to the origin, $a= 0 \text{ or } a>0,\ 2_\alpha^\sharp<p<2_\alpha^*.$ We further assume 
    \begin{align}
        c<\sup _{u \in S_1}\left(\min\left \{\Lambda_{1,u}^\frac{N-2}{\alpha+2},\frac{\alpha+2}{2N+2\alpha}\mathcal{S}_H^\frac{N+\alpha}{\alpha+2}\Lambda_{2,u}^{-1}\right \} \right), \label{aj} 
    \end{align}
    for $a=0$, and
    \begin{align}
        c<\underset{\xi\in(0,1)}{\max}&\left\{\min\left\{\left (\frac{1}{2}-\frac{1}{2\eta_pp}\right )\xi ^{2_\alpha}\mathcal{S}_H^{\frac{N+\alpha}{\alpha+2} }\Lambda_{2,u}^{-1} \right.\right.,\nonumber \\ 
        &\left (\frac{1-\xi}{2a\eta_p\mathcal{C}(N,\mu)\mathcal{C}_{ 2N/\left ( N+\alpha \right ) }^{2p}}\right )^{\frac{2}{p\eta_p-p-2}}\left ( \left (\frac{1}{2}-\frac{1}{2\eta_pp}\right )\Lambda_{2,u}^{-1} \right )^{\frac{2\left ( p\eta_p-1\right ) }{p\eta_p-p-2} }\bigg\}\bigg\},\label{ak} 
    \end{align}
    for $a>0$, where
    \begin{align*}
        &\Lambda_{1,u}:=\frac{\int_{\Omega}\left|\nabla u\right|^2dx}{  {\mathbb{D}}(u^+)},\ \ \ 
        \Lambda_{2,u}:=\frac{1}{2} \int_{\Omega}\left|\nabla u\right|^2dx,\\
        &\Lambda_{3,u}:=\frac{\int_{\Omega}\left|\nabla u\right|^2dx}{a\eta_p {\mathbb{B}_p}(u^+)},\ \ \ 
        \Lambda_{4,u}:=\frac{2p}{|a|{\mathbb{B}_p}(u^+)}.
    \end{align*}
    Then the value 
    $$
    \nu _c:=\inf _{u \in \mathcal{G}} E(u) \in
    $$
    \begin{equation}
        \left\{\begin{aligned}
            &\left ( 0, \frac{\alpha+2}{2N+2\alpha}\mathcal{S}_H^{\ \ \frac{N+\alpha}{\alpha+2}} \right ) \text{ if } a= 0,\\
            &\left(0,\left (  \frac{1}{2}-\frac{1}{2\eta_pp} \right )\underset{\xi\in\left ( 0,1 \right ) }{\max}\min \left \{\xi ^{2_\alpha}\mathcal{S}_H^{\frac{N+\alpha}{\alpha+2} },\left ( \frac{1-\xi }{2a\eta_pC(N,\mu)\mathcal{C}_{ 2N/\left ( N+\alpha \right ) }^{2p}c^{\frac{p(1-\eta_p)}{2}}} \right )^\frac{1}{p\eta_p-1} \right \}\right)\\
            &\qquad \qquad \qquad \qquad \qquad \qquad \qquad \qquad \qquad \qquad \qquad \qquad  \qquad \text{ if }a>0,\ 2_\alpha^\sharp<p<2_\alpha^*.\nonumber \\
        \end{aligned}
        \right.
\end{equation}
    and $\nu_c$ is achieved in $\mathcal{G}$ by some $u_c \in S_c^+$. Furthermore, $u_c > 0$ is a normalized
    ground state to \eqref{aa} with Lagrange multiplier $\lambda_c$ satisfying $\lambda_c > 0$ when $a = 0$ and $\lambda_c < \lambda_1(\Omega)$ when $a \ge 0$, where $\lambda_1(\Omega)$is the first eigenvalue of $-\Delta$ on $\Omega$ with Dirichlet boundary condition.
\end{theorem}

\begin{remk}
    To obtain the normalized ground state to \eqref{aa}, we first prove there exists $c_0>0$ such that $\mathcal{G}$ is not empty when $c<c_0$, and all minimizing sequences to $\nu_c$ is compact. Recovering the compactness is a challenge to Brézis-Nirenberg type critical problem, and we use a concentration compactness argument to overcome it. In addition, in order to guarantee that the minimizers to $\nu_c$ are critical points of $E$ restricted on $S_c$, it is also necessary to show that they do not belong to the boundary of $\mathcal{G}$.
\end{remk}

The solutions $u_c$ found in Theorem \ref{C} appear as a local minimizer of $E$ on $S_c$ and we recall that $E$ is unbounded from below $S_c$, which indicates that $E$ has a so-called mountain-pass geometry on $S_c$. More precisely, we prove that there exists an element $v\in S_c\setminus\mathcal{G}$ and define
$$
m(c):=\inf _{\gamma  \in \Gamma} \sup _{t \in[0,1]} E(\gamma (t))>E(u_c), 
$$
where
\begin{align}
    \Gamma:=\left\{\gamma \in C\left([0,1], S_c^+\right) \text { s.t. } \gamma (0) =u_c, \gamma (1) =v \right\}.\label{ak1}
\end{align}

\begin{theorem}\label{E}
    Under the assumptions of Theorem \ref{C}, for $a=0$, we further assume that one of the following holds:
    \begin{itemize}
	    \item[$\bullet$] $N\in \left \{3,4\right \}$,
	    \item[$\bullet$] $N\ge 5$ and $ \alpha \in [N-2,N)$.
    \end{itemize}
    For $a>0$, we further assume that one of the following holds:
    \begin{itemize}
	    \item[$\bullet$] $N\ge 6$, $\alpha \in(0,N-2]$, and $2_\alpha^\sharp<p<2_\alpha^*$,
	    \item[$\bullet$] $N=5$, $ \alpha \in (0,5)$, and $ \max\left \{\frac{7+2\alpha}{6},2_\alpha^\sharp\right \}<p<2_\alpha^*$,
	    \item[$\bullet$] $N=4$, for $\alpha\in (0,2)$, $\frac{3+\alpha}{2}<p<2_\alpha^*$ and for $\alpha\in [2,4)$, $2_\alpha^\sharp<p<2_\alpha^*$,
	    \item[$\bullet$] $N=3$, for $\alpha\in (0,1)$, $\frac{5+2\alpha}{2}<p<2_\alpha^*$ and for $\alpha\in [1,3)$, $2_\alpha^\sharp<p<2_\alpha^*$.
    \end{itemize}
    Then the equation \eqref{aa} has the second positive solution $\widetilde{u_c}\ne u_c$ at the level $m(c)\in (\nu_c,\nu_c+\frac{\alpha +2}{2N+2\alpha }\mathcal{S}_H^{\frac{N+\alpha }{\alpha+2}})$.
\end{theorem}

\begin{remk}
    One of the main changes of searching for mountain-pass solution is establishing a bounded (PS) sequence. The method to obtain a bounded (PS) sequence on entire space fails due to bounded domain is not scaling invariant. We overcome this difficulty using monotonicity trick, which originates from Struwe's work \cite{MR2431434,MR926524} and is further developed by Jeanjean into a general result \cite{MR1718530}. More precisely, we consider the family of functionals $E_\theta: H_0^1(\Omega) \to \mathbb{R} $ defined by
    \begin{align*}
        E_\theta(u)=\frac{1}{2}\int_{\Omega}|\nabla u|^2 dx-\theta\left (\frac{1}{2\cdot 2_\alpha^*}  {\mathbb{D}}(u^+)+\frac{a}{2p}{\mathbb{B}_p}(u^+)\right ).
    \end{align*}
    These functionals possess a MP geometry and we deduce from \cite{MR1718530} that there exists a sequence $\theta_j \in [\frac{1}{2},1]$ such that $\theta_j \to 1$ as $n \to \infty$ and $E_{\theta_j}$ has a bounded PS sequence $\{u_n^j\}$ at level $m_\theta$. We can prove that $\{u_n^j\}$ converges weakly to a nontrivial critical point $u_j$ of $E_{\theta_j}$. Pohozaev type identity yields the boundedness of $\{u_n\}$, which is the bounded PS sequence we are looking for.   
\end{remk}

\begin{remk}
    Another difficulty is recovering the compactness, and we use a concentration compactness argument to overcome it. The estimation of the value of the level $m(c)$  played a significant role during this process.
\end{remk}

The paper is organized as follows. In Sect. 2, we prove $\inf _{u \in \mathcal{G}} E(u)$ is attainable, and obtain a normalized ground state. In Sect. 3, we prove the existence of normalized mountain pass solution.

\section{Positive normalized ground state}

\begin{proposition}\label{B}
    Let $\Omega$ be smooth, bounded, and star-shaped with respect to the origin.
    If $u$ is a critical point of $\left.E\right|_{S_c^+}$, then $u \in \mathcal{G}$. Further, we assume that $a=0$ or $a>0,\ 2_\alpha^\sharp<p<2_\alpha^*$, and have\\
	$\mathrm{(i)}$ Any sequence ${u_n}\subset \mathcal{G}$ satisfying $\limsup _{n \rightarrow \infty} E\left(u_n\right)<\infty$ is bounded in $H_0^1(\Omega),$\\
    $\mathrm{(ii)}$ If $c$ satisfies \eqref{aj}-\eqref{ak}, then $\mathcal{G} \neq \emptyset$, and it holds
\begin{align}
    0<\inf _{u \in \mathcal{G}} E(u)<\inf _{u \in \partial \mathcal{G}} E(u),\label{al}
\end{align}
    $\mathrm{(iii)}$
    If $u\in \mathcal{G}$, there exists unique $t_u > 1$ such that $u^{t_u} \in \partial \mathcal{G}$.
\end{proposition}

\begin{proof}
    Since $\Omega$ is star-shaped with respect to the origin, using the Pohozaev identity \eqref{ae1}, we have $u \in \mathcal{G}$ for any $u$ with ${\left.E'\right|_{S_c^+}}(u)=0$. \\

    (i) Let $\{u_n\} \subset \mathcal{G}$ satisfy $\underset{n \rightarrow \infty}{\limsup}E\left(u_n\right)<\infty $. For $a= 0$, using \eqref{af}, we get
    $$
    \limsup _{n \rightarrow \infty} \int_{\Omega}|\nabla u|^2 dx \le \frac{2N+2\alpha}{\alpha+2} \limsup _{n \rightarrow \infty}E\left(u_n\right)<\infty
    $$
    yielding $\{u_n\}$ is bounded in $H_0^1(\Omega)$.
    Notice that $p\eta_p> 1$ when
    $2_\alpha^\sharp<p<2_\alpha^*$. For $a > 0,\ 2_\alpha^\sharp<p<2_\alpha^*$, using \eqref{ah} we get 
    $$
    \limsup _{n \rightarrow \infty} \int_{\Omega}|\nabla u|^2 dx \le (\frac{1}{2}-\frac{1}{2p\eta_p} )\limsup _{n \rightarrow \infty}E\left(u_n\right)<\infty
    $$
    yielding $\{u_n\}$ is bounded in $H_0^1(\Omega)$.

    (ii) When $a=0$, for $u \in S_1^+,\ c<\Lambda_{1,u}^\frac{N-2}{\alpha+2}$ implies that 
    \begin{align*}
        \int_{\Omega}|\nabla(\sqrt{c} u)|^2 d x  =c \int_{\Omega}|\nabla u|^2 dx >c^{2_\alpha^*}\mathbb{D}(\sqrt{c}u^+) =\int_{\Omega }(I_\alpha*|\sqrt{c} u^+|^{2_\alpha^*})|\sqrt{c} u^+|^{2_\alpha^*}dx. 
    \end{align*}
    Hence, $\sqrt{c}u \in \mathcal{G}$, yielding that $\mathcal{G}$ is not empty. Further, using \eqref{ag} one gets 
    $$
    \inf _{v \in \partial \mathcal{G}} \int_{\Omega}|\nabla v|^2 dx \geq \mathcal{S}_H^\frac{N+\alpha}{\alpha+2}.
    $$
    Then using \eqref{af} we have 
    $$
    \inf _{v \in \partial \mathcal{G}}E(v)\ge \frac{\alpha+2}{2N+2\alpha}\int_{\Omega}|\nabla v|^2 dx\ge \frac{\alpha+2}{2N+2\alpha}\mathcal{S}_H^\frac{N+\alpha}{\alpha+2}.
    $$
    For $u\in S_1^+,\ c<\frac{\alpha+2}{2N+2\alpha}\mathcal{S}_H^\frac{N+\alpha}{\alpha+2}\Lambda_{2,u}^{-1}$ implies that 
    $$
    E(\sqrt{c} u)<\frac{c}{2} \int_{\Omega}|\nabla u|^2 d x<\frac{\alpha+2}{2N+2\alpha}\mathcal{S}_H^\frac{N+\alpha}{\alpha+2} \leq \inf _{v \in \partial \mathcal{G}} E(v).
    $$
    When \eqref{aj} holds, we can take some $u\in S_1^+$ such that $\sqrt{c}u \in \mathcal{G}$ and 
    $$
    \inf _{v \in \mathcal{G}}E(v)\le E(\sqrt{c}u)<\inf _{v \in \partial \mathcal{G}}E(v).
    $$
    Moreover, using \eqref{af} again we have   
    $$
    \inf _{v \in \mathcal{G}}E(v)\ge \frac{\alpha+2}{2N+2\alpha}\int_{\Omega}|\nabla v|^2 dx \ge \frac{\alpha+2}{2N+2\alpha}\lambda_1c>0.
    $$
    Hence, \eqref{am} holds true.\\

    When $a>0$ and $2_\alpha^\sharp<p<2_\alpha^*$, for any $u\in S_1^+$ and some $\tau \in (0,1)$, combining $c<(\tau \Lambda_{1,u})^\frac{N-2}{\alpha+2}$ and $c<((1-\tau)\Lambda_{3,u})^{1/(p-1)}$ implies that 
    \begin{align*}
        \int_{\Omega}|\nabla(\sqrt{c} u)|^2 d x  &=c \int_{\Omega}|\nabla u|^2 dx=c\tau \int_{\Omega}|\nabla u|^2 dx+c(1-\tau)\int_{\Omega}|\nabla u|^2 dx\\
        &>c^{2_\alpha^*}{\mathbb{D}}(u^+)+a\eta_pc^p\mathbb{B}_p(u^+)={\mathbb{D}}(\sqrt{c}u^+)+a\eta_p\mathbb{B}_p(\sqrt{c}u^+).
    \end{align*}
    Hence, $\sqrt{c}u \in \mathcal{G}$, yielding that $\mathcal{G}$ is not empty. Further, using \eqref{ai} one gets for any $\xi \in (0,1)$,
    $$
    \inf _{v \in \partial \mathcal{G}} \int_{\Omega}|\nabla v|^2 dx\geq \xi ^{\frac{N-2}{\alpha+2}}\mathcal{S}_H^{\frac{N+\alpha}{\alpha+2} }
    $$
    or
    $$
    \inf _{v \in \partial \mathcal{G}} \int_{\Omega}|\nabla v|^2 dx \geq \left ( \frac{1-\xi }{a\eta_p\mathcal{C}_G(N,\alpha ,p)c^{p(1-\eta_p)}} \right )^\frac{1}{p\eta_p-1}. 
    $$
    Then using \eqref{ah} we have for any $\xi \in (0,1)$,
    $$
    \inf _{v \in \partial \mathcal{G}}E(v)\ge\left (  \frac{1}{2}-\frac{1}{2\eta_pp} \right )\min \left \{\xi ^{\frac{N-2}{\alpha+2}}\mathcal{S}_H^{\frac{N+\alpha}{\alpha+2} },\ \left ( \frac{1-\xi }{a\eta_p\mathcal{C}_G(N,\alpha ,p)c^{p(1-\eta_p)}} \right )^\frac{1}{p\eta_p-1}\right \} .
    $$
    For $u\in S_1^+$, from 
    \begin{align*}
        &c<\underset{\xi\in(0,1)}{\max}\left\{\min\left\{\left (\frac{1}{2}-\frac{1}{2\eta_pp}\right )\xi ^{\frac{N-2}{\alpha+2}}\mathcal{S}_H^{\frac{N+\alpha}{\alpha+2} }\Lambda_{2,u}^{-1} ,\right.\right.\\
        &\left (\frac{1-\xi}{a\eta_p\mathcal{C}_G(N,\alpha ,p)}\right )^{\frac{2}{p\eta_p-p-2}}\left ( \left (\frac{1}{2}-\frac{1}{2\eta_pp}\right )\Lambda_{2,u}^{-1} \right )^{\frac{2\left ( p\eta_p-1\right ) }{p\eta_p-p-2} }\bigg\}\bigg\},
    \end{align*}
    we know that exist some $\xi_0 \in (0,1)$ such that 
    \begin{align*}
        &E(\sqrt{c} u)<\frac{c}{2} \int_{\Omega}|\nabla u|^2 d x\\
        <&\left (\frac{1}{2}-\frac{1}{2\eta_pp}\right ) \min \left \{\xi_0 ^{\frac{N-2}{\alpha+2}}\mathcal{S}_H^{\frac{N+\alpha}{\alpha+2} },\ \left ( \frac{1-\xi_0 }{a\eta_p\mathcal{C}_G(N,\alpha ,p)c^{p(1-\eta_p)}} \right )^\frac{1}{p\eta_p-1}\right \} \\
        \le& \inf _{v \in \partial \mathcal{G}}E(v).
    \end{align*}
    When \eqref{ak} holds, we can take some $u\in S_1^+$ such that $\sqrt{c}u \in \mathcal{G}$ and 
    $$
    \inf _{v \in \mathcal{G}}E(v)\le E(\sqrt{c}u)<\inf _{v \in \partial \mathcal{G}}E(v).
    $$
    Moreover, using \eqref{ah} again we have   
    $$
    \inf _{v \in \mathcal{G}}E(v)\ge \left (\frac{1}{2}-\frac{1}{2\eta_pp}\right ) \inf _{v \in \mathcal{G}}\int_{\Omega}|\nabla v|^2 dx \ge \left (\frac{1}{2}-\frac{1}{2\eta_pp}\right )\lambda_1c>0.
    $$
    Hence, \eqref{al} holds true.

    (iii) Recall that $u^t(x):=t^{N/2}u(tx)$, with $t\ge1.$ For $u \in S_c^+$, we define 
    $$
    \phi(t):=E(u^t)=\frac{t^2}{2}\int_{\Omega}|\nabla u|^2 d x-\frac{t^{2\cdot2_\alpha^*}}{2\cdot 2_\alpha^*}  {\mathbb{D}}(u^+)-\frac{a}{p}t^{2\eta_pp}{\mathbb{B}_p}(u^+).
    $$
    the derivative of $\phi(t)$ satisfies
    \begin{align*}
        \phi'(t)&=t\int_{\Omega}|\nabla u|^2 d x-t^{2\cdot2_\alpha^*-1}  {\mathbb{D}}(u)-2a\eta_pt^{2\eta_p-1}{\mathbb{B}_p}(u)\\
        &=t^{-1}\left (\int_{\Omega}|\nabla u^t|^2 dx-
        {\mathbb{D}}(u^t)-
        a\eta_p {\mathbb{B}_p}(u^t)\right ). 
    \end{align*}
    Hence, $u^t\in \mathcal{G}$ is equivalent to $\phi'(t)>0$ and $u^t\in \partial \mathcal{G}$ is equivalent to $\phi'(t)=0$. Since $2_\alpha^*>\max\{1,p\eta_p\},$ direct calculations show that for a unique $t_u>0,\ \phi'(t_u)=0, \phi'(t)>0$ in $(0,t_u)$ and $\phi'(t)<0$ in $(t_u,\infty),\ u\in \mathcal{G}$ implies that $\phi'(1)>0$, and so $t_u>0$, and so $t_u>1$. Further, $\phi'(t_u)=0$ shows $u^{t_u}\in \partial \mathcal{G}$.      
\end{proof}

Before proving Theorem \ref{C}, we introduce the Brezis-Lieb type lemma for the nonlocal term, see \cite[Lemma 2.2]{MR3817173} for the critical case and \cite[Lemma 3.4]{MR2088936} for the subcritical case.

\begin{lemma}\label{L}
    Let $N\ge3,\ 0<\alpha<N$ and $2_\alpha \le p \le2_\alpha^*$, if $\{u_n\}$ is a bounded sequence in $L^{2^*}(\mathbb{R}^N)$ such that $u_n\to u$ a.e. in $\mathbb{R}^N$ as $n\to \infty$, then we have 
    $$
    \int_{\mathbb{R}^N}(I_\alpha*|u_n|^{p})|u_n|^{p}dx=\int_{\mathbb{R}^N}(I_\alpha*|u_n-u|^{p})|u_n-u|^{p}dx+\int_{\mathbb{R}^N}(I_\alpha*|u|^{p})|u|^{p}dx+o_n(1).
    $$  
\end{lemma}

\begin{proof}[proof of Theorem \ref{C}]
    The range of $\nu_c$ is given in Proposition \ref{B}. We prove $\nu_c$ can be achieved in $\mathcal{G}$ below. Let $\{u_n\} \subset \mathcal{G}$ be a minimizing sequence of $\nu_c$. Proposition \ref{B} $(i)$ yields that $\left(u_n\right)$ is bounded in $H_0^1(\Omega)$.

    \textbf{Claim:} $\{u_n\}$ is bounded away from $\partial \mathcal{G}$.

    Suppose on the contrary that there exists $\{v_n\}\subset \partial \mathcal{G}$ such that, up to subsequences, $u_n-v_n \rightarrow 0$ in $H_0^1(\Omega)$. $\{v_n\}$ is bounded in $H_0^1(\Omega)$ due to the boundedness of $\{u_n\}$.     
    $$
    \inf _{u \in \mathcal{G}} E(u)=\nu_c=\lim _{n \rightarrow \infty} E\left(u_n\right)=\lim _{n \rightarrow \infty} E\left(v_n\right) \geq \inf _{v \in \partial \mathcal{G}} E(v),
    $$
    contradicting \eqref{ah}. This completes the claim.

    By Ekeland variational principle, we may assume that $\left.E^{\prime}\right|_{S_c^+}\left(u_n\right)=\left.E^{\prime}\right|_{\mathcal{G}}\left(u_n\right) \rightarrow 0$ as $n \rightarrow \infty$. From the boundedness of $\{u_n\}$, we pass to a subsequence satisfying
    \begin{align*}
        & u_n \rightharpoonup \tilde{u}_c \quad \text { weakly in } H_0^1(\Omega), \\
        & u_n \rightharpoonup \tilde{u}_c \quad \text { weakly in } L^{2^*}(\Omega), \\
        & u_n \rightarrow \tilde{u}_c \quad \text { strongly in } L^r(\Omega) \text { for } 2<r<2^*, \\
        & u_n \rightarrow \tilde{u}_c \quad \text { a.e. on } \Omega.
    \end{align*}
    Since $\left.E'\right|_{S_c^+}\left(u_n\right)\to 0 \text{ as } n\to \infty$, there exists $\lambda_n$ such that $E'(u_n)-\lambda_nu_n^+ \to 0 \ as \ n \to \infty$. Let $\lambda_c$ be the Lagrange multiplier corresponding to $u_c$, and $w_n=u_n-u_c$. For some $\psi \in H_0^1(\Omega)$ with $\int_{\Omega}u_c^+\psi dx \ne 0$, we obtain 
    $$
    \lambda_n=\frac{1}{\int_{\Omega}|u_n^+\psi| dx} \left (  \left \langle E'(u_n),\psi \right \rangle+o_n(1) \right )\overset{n\to\infty}{\longrightarrow}\frac{1}{\int_{\Omega}|u_c^+\psi| dx} \left (  \left \langle E'(u_c),\psi \right \rangle+o_n(1) \right )=\lambda _c.
    $$
    Using the Brézis-Lieb Lemma (see [Lemma \ref{L}]) and the facts that
    $$
    E'(u_n)-\lambda_nu_n^+ \to 0,\ \ \ \  E'(u_c)-\lambda_cu_c^+ = 0,
    $$
    we get
    $$
    \int_{\Omega}|\nabla w_n|^2 dx={\mathbb{D}}(w_n)+o_n(1).
    $$
    Hence we assume that $\int_{\Omega}|\nabla w_n|^2 dx\to l\ge 0,\ {\mathbb{D}}(w_n)\to l\ge 0$. Using Brézis-Lieb Lemma again, we deduce that
    $$
    E(u_n)=E(u_c)+E(w_n)+o_n(1).
    $$
    Note that 
    $$
    E(w_n)=\frac{1}{2}\int_{\Omega}|\nabla w_n|^2 dx-{\mathbb{D}}(w_n)+o_n(1)=\frac{\alpha+2}{2N+2\alpha}l+o_n(1),    
    $$
    that implies $E(w_n)\ge o_n(1)$, and so $\nu_c=\lim_{n \to \infty} E(u_n)\ge E(u_c).$ On the other hand, it can be verified that $u_c\in S_c^+$ is a critical point of $E$ constrained on $S_c^+$. Since 
    $u_c \in \mathcal{G}$,we have $E(u_c)\geq \nu_c$. Hence, $E(u_c)=\nu_c$.    

    Next, we prove $u_c$ is a positive normalized ground state to \eqref{aa}. By Lagrange multiplier principle, $u_c$ satisfies 
    $$
    -\Delta u_c=\lambda_c u_c^++\mathbb{D}(u_c^+)+a\mathbb{B}_p(u_c^+)
    $$
    for some $\lambda_c$. Multiplying $u_c^-$ and integrating on $\Omega$, we obtain $\int_{\Omega}|\nabla u_c^+|^2 dx=0$, implying that $u_c^-= 0$ and hence $u_c\ge 0$. Since $u_c\ne 0$, we conclude $u_c>0$ for all $x$ by the strong maximum principle, and $\int_{\Omega}|u_c|^2 dx=\int_{\Omega}|u_c^+|^2 dx=c$. Then, we show $u_c$ is a normalized ground state. Since $u_c$ is positive, we have $E(u_c)=\widetilde{E}(u_c) $. If $v\in S_c$ satisfy $\left.E\right|_{S_c^+}'\left(v\right)= 0$, we have $|v|\in S_c^+$. Moreover, $v$ satisfies the following identity:
    \begin{align*}
        \int_{\Omega}\left|\nabla v\right|^2 dx-\frac{1}{2}\int_{\partial \Omega} \left|\nabla v\right|^2\sigma\cdot n d\sigma ={\mathbb{D}}(v)+a\eta_p{\mathbb{B}_p}(v).
    \end{align*}
    and $\int_{\Omega}|\nabla v|^2 dx=\int_{\Omega}|\nabla |v||^2 dx$. Hence, we have $|v|\in \mathcal{G}$, and then 
    $$
    \widetilde{E}(|v|)=E(|v|)\ge \nu_c=E(u_c)= \widetilde{E}(u_c).
    $$  
    Using $\int_{\Omega}|\nabla v|^2 dx=\int_{\Omega}|\nabla |v||^2 dx$ again we have $\widetilde{E}(u_c)\le \widetilde{E}(|v|)= \widetilde{E}(v).$ The arbitrariness of $v$ yields that $u_c$ is a normalized ground state solution to \eqref{aa}. 

    Finally, we estimate the range of $\lambda_c$. For $a= 0$, multiplying $u_c$ and integrating on $\Omega$ for equation \eqref{aa} of $u_c$, we obtain 
    \begin{align*}
        \lambda_cc+{\mathbb{D}}(u_c)=\int_{\Omega}|\nabla u_c|^2 dx>{\mathbb{D}}(u_c).
    \end{align*}
    Hence, $\lambda_c>0$ when $a= 0$. Let $e_1$ be the corresponding positive unit eigenfunction of $\lambda_1(\Omega)$.  Multiplying $e_1$ and integrating on $\Omega$ for equation \eqref{aa} of $u_c$, we obtain 
    \begin{align*}
        &\lambda_c\int_{\Omega}u_ce_1 dx+\int_{\Omega}(I_\alpha*|u_c|^{2_\alpha^*})|u_c|^{2_\alpha^*-2}u_ce_1dx+a\int_{\Omega}(I_\alpha*|u_c|^p)|u_c|^{p-2}u_ce_1dx\\
        =&\int_{\Omega}\nabla u_c\nabla e_1 dx=\lambda_1(\Omega)\int_{\Omega}u_ce_1 dx
    \end{align*} 
    implying that $\lambda_c<\lambda_1(\Omega)$ when $a\ge 0$. The proof is complete.
\end{proof}

\section{The second positive normalized solution of M-P type}

For $\theta \in [1/2,1]$, we consider the family of functionals $E_\theta: H_0^1(\Omega) \to \mathbb{R} $ defined by
\begin{align*}
    E_\theta(u)=\frac{1}{2}\int_{\Omega}|\nabla u|^2 dx-\theta\left (\frac{1}{2\cdot 2_\alpha^*}  {\mathbb{D}}(u^+)+\frac{a}{2p}{\mathbb{B}_p}(u^+)\right ).
\end{align*}
Similar to the case that $\theta=1$, we introduce
$$
\mathcal{G}_\theta:=\left\{u \in S_c^+: \int_{\Omega}\left|\nabla u\right|^2 dx>\theta{\mathbb{D}}(u^+)+\theta a\eta_p{\mathbb{B}_p}(u^+)\right\},
$$
for $a\ge0$. It is easy to see $\mathcal{G} \subset \mathcal{G}_\theta$ for $\theta<1$, implying $\mathcal{G}_\theta$ is not  empty when $\mathcal{G}\ne \emptyset$. We further set 
$$
\nu_{c,\theta}:=\underset{\mathcal{G}_\theta}{\inf}E_\theta.
$$ 

In order to show $E_\theta$ possesses a bounded PS sequence, we use the abstract results of \cite{MR1718530}.

\begin{theorem}\label{I}
    Let $X$ be a Banach space equipped with a norm $\left \| \cdot  \right \|_X$ and let $J \subset \mathbb{R}^{+}$ be an interval. We consider a family $\left \{E_\theta\right \}_{\theta \in I}$ of $C^1$-functionals on $X$ of the form
    $$
    E_\theta(u)=A(u)-\theta B(u), \quad \theta \in J,
    $$
    where $B(u) \geq 0, \forall u \in X$ and such that either $A(u) \rightarrow \infty$ or $B(u) \rightarrow \infty$ as $\|u\| \rightarrow$ $\infty$. We assume there are two points $\left(v_1, v_2\right)$ in $X$(independent of $\theta$ ) such that setting
    $$
    \Gamma=\left\{\gamma \in C\left([0,1],X\right)| \gamma(0)=v_1, \gamma(1)=v_2\right\}
    $$
    there holds, $\forall \tau \in I$,
    $$
    m_\theta:=\inf _{\gamma \in \Gamma} \sup _{t \in[0,1]} E_\theta(\gamma(t))>\max \left\{E_\theta\left(v_1\right), E_\theta\left(v_2\right)\right\}.
    $$
    Then, for almost every $\theta \in I$, there is a sequence $\left\{v_n\right\} \subset X$such that\\
    (i) $\left\{v_n\right\}$ is bounded in $X$, (ii) $E_\theta\left(v_n\right) \rightarrow m_\theta$, (iii) $\left.E_\theta^{\prime}\right|_{S_c^{+}}\left(v_n\right) \rightarrow 0$ in the dual $X^{\prime}$ of $X$.
\end{theorem}
 
We shall use Theorem \ref{I} with $X=H_0^1(\Omega),\ \left \| \cdot  \right \|_X=\left \| \cdot  \right \|_{H_0^1(\Omega)},\ J=\left [ \frac{1}{2},1\right ] $. Set the range of $\gamma$ is on $S_c^+$. In order to prove the existence of bounded PS sequence to $E_\theta$, we need to prove the MP geometry of $E_\theta$ for each $\theta \in [1/2,1]$. Before that, we remark
\begin{lemma}\label{F}
    Let $\Omega$ be smooth, bounded, and star-shaped with respect to the origin, $a=0$ or $a>0,\ 2_\alpha^\sharp<p<2_\alpha^*$, and let $c$ satisfy \eqref{aj}-\eqref{ak}. Then we have 
    \begin{align}
        &\underset{\theta\to 1^-}{\lim}E_\theta=\underset{\partial \mathcal{G}}{\inf}E, \label{am}\\ 
        &\underset{\theta\to 1^-}{\lim\inf}E>\underset{\mathcal{G}}{\inf}E,\label{an}\\ 
        &\underset{\theta\to 1^-}{\lim}\nu_{c,\theta}=\nu_c.\label{ao}
    \end{align}
\end{lemma}

\begin{proof}
    The proof is similar to the one of \cite[Lemma 4.1]{song2024}, and we ignore the details. 
\end{proof}

The following lemma ensures that $E_\theta$ has MP geometry.
\begin{lemma}[Uniform M-P geometry.]\label{G}
    Under the assumptions of Lemma \ref{F}, there exist $\epsilon \in (0, 1/2)$ and $\delta > 0$ independent of $\theta$ such that 
    \begin{align}
        E_\theta\left(u_c\right)+\delta<\inf _{\partial \mathcal{G}} E_\theta,\  \forall \theta \in(1-\epsilon, 1],\label{ap}
    \end{align}
    and there exists $v \in S_c^{+} \backslash \mathcal{G}$ such that
    $$
    m_\theta:=\inf _{\gamma \in \Gamma} \sup _{t \in[0,1]} E_\theta(\gamma(t))>E_\theta\left(u_c\right)+\delta=\max \left\{E_\theta\left(u_c\right), E_\theta(v)\right\}+\delta,
    $$
    where
    $$
    \Gamma:=\left\{\gamma \in C\left([0,1], S_c^{+}\right): \gamma(0)=u_c, \gamma(1)=v\right\}
    $$
    is independent of $\theta$.
\end{lemma}

\begin{proof}
    Note that $\lim _{\theta \rightarrow 1^{-}} E_\theta\left(u_c\right)=E\left(u_c\right)=\nu_c$. Then using \eqref{am} and Proposition \ref{B} (ii) we have
    $$
    \lim _{\theta \rightarrow 1^{-}} E_\theta\left(u_c\right)=\nu_c<\inf _{\mathcal{G}} E=\lim _{\theta \rightarrow 1^{-}} \inf _{\mathcal{G}} E_\theta .
    $$
    We may choose $2 \delta=\inf _{\mathcal{G}} E-\nu_c$ and $\epsilon$ small enough, and \eqref{ap} holds.
    For $w \in S_c^{+}$, we recall that
    $$
    w^t=t^{\frac{N}{2}} w(t x) \in S_c^{+}, t \geq 1.
    $$
    For $a\ge 0$ and $2_\alpha<p<2_\alpha^*$, we have as $t \rightarrow \infty$,
    \begin{align*}
        E_\theta\left(w_t\right) =&\frac{1}{2} \int_{\Omega}\left|\nabla w_t\right|^2 d x-\theta( \frac{1}{2\cdot 2_\alpha^*}{\mathbb{D}}(w_t^+)+\frac{a}{2p}{\mathbb{B}_p}(w_t^+))\\
        \le& \frac{t^2}{2}\int_{\Omega}|\nabla w|^2 dx-\frac{1}{2}(\frac{t^{2\cdot 2_\alpha^*}}{2\cdot 2_\alpha^*}{\mathbb{D}}(w^+)+\frac{a}{p}t^{2p\gamma _p}{\mathbb{B}_p}(w^+))\\
        \rightarrow&-\infty \text { uniformly w.r.t. } \theta \in\left[\frac{1}{2}, 1\right].
    \end{align*}

    Take $v=w^t$ with $t$ large enough such that $E_\theta(v)<E_\theta\left(u_c\right)$. Also, using similar arguments to the proof to Proposition \ref{B} (iii), we can assume $v \notin \mathcal{G}$. Then for any $\gamma \in \Gamma$, there exists $t^* \in(0,1)$ such that $\gamma\left(t^*\right) \in \partial \mathcal{G}$. Hence,
    $$
    \inf _{\gamma \in \Gamma} \sup _{t \in[0,1]} E_\theta(\gamma(t)) \geq \inf _{u \in \partial \mathcal{G}} E_\theta(u) .
    $$
    By \eqref{ap} we complete the proof.

\end{proof}

\begin{lemma}\label{H}
    Under the assumptions of Lemma \ref{G}, then $\underset{\theta\to 1^-}{\lim}m_\theta=m_1$.
\end{lemma}

\begin{proof}
    The proof of Lemma \ref{H} is quiet similar to the one of \cite[Lemma 4.3]{song2024} and we ignore details.
\end{proof}

\begin{proposition}\label{J}
    (Positive solutions for almost every $\theta$). Under the assumptions of Lemma \ref{G}, when $a=0$, we further assume that one of the following holds:
    \begin{itemize}
	    \item[$\bullet$] $N\in \left \{3,4\right \}$,
	    \item[$\bullet$] $N\ge 5$ and $ \alpha \in [N-2,N)$.
    \end{itemize}
    For $a>0$, we further assume that one of the following holds:
    \begin{itemize}
	    \item[$\bullet$] $N\ge 6$, $\alpha \in(0,N-2]$, and $2_\alpha^\sharp<p<2_\alpha^*$,
	    \item[$\bullet$] $N=5$, $ \alpha \in (0,5)$, and $ \max\left \{\frac{7+2\alpha}{6},2_\alpha^\sharp\right \}<p<2_\alpha^*$,
	    \item[$\bullet$] $N=4$, for $\alpha\in (0,2)$, $\frac{3+\alpha}{2}<p<2_\alpha^*$ and for $\alpha\in [2,4)$, $2_\alpha^\sharp<p<2_\alpha^*$,
	    \item[$\bullet$] $N=3$, for $\alpha\in (0,1)$, $\frac{5+2\alpha}{2}<p<2_\alpha^*$ and for $\alpha\in [1,3)$, $2_\alpha^\sharp<p<2_\alpha^*$.
    \end{itemize}
    Then for almost every $\theta \in[1-\epsilon, 1]$ where $\epsilon$ is given by Lemma \ref{G} , there exists $a$ critical point $u_\theta$ of $E_\theta$ constrained on $S_c^{+}$at the level $m_\theta$, which solves
    \begin{equation}
        \left\{\begin{aligned}
            &-\Delta u =\lambda u+\theta (I_\alpha*|u|^{2_\alpha^*})|u|^{2_\alpha^*-2}u+a\theta(I_\alpha*|u|^p)|u|^{p-2}u,\ x\in\Omega,,\\
            &u_\theta=0 \text { on } \partial \Omega,\ u_\theta>0,\\
        \end{aligned}
        \right.\label{aq}
    \end{equation}
    when $a\ge 0$ for some $\lambda_\theta$. Furthermore, $u_\theta \in \mathcal{G}_\theta$.
\end{proposition}

\begin{proof}
    We set 
    \begin{align*}
    A(u)=\frac{1}{2} \int_{\Omega}|\nabla u|^2 d x, \ \ B(u)=\frac{1}{2\cdot 2_\alpha^*}  {\mathbb{D}}(u^+)+\frac{a}{2p}{\mathbb{B}_p}(u^+).
    \end{align*}
    Together with Lemma \ref{G}, we have for almost every $\theta \in[1-\epsilon, 1]$, there exists a bounded (PS) sequence $\left\{u_n\right\} \subset S_c^{+}$ satisfying $E_\theta\left(u_n\right) \rightarrow m_\theta$ and $\left.E_\theta'\right|_{S_c^{+}}\left(u_n\right) \rightarrow 0$ as $n \rightarrow \infty$. From the boundedness of $\{u_n\}$, we may assume, going if necessary to a subsequence
    \begin{align*}
        & u_n \rightharpoonup u_\theta \quad \text { weakly in } H_0^1(\Omega), \\
        & u_n \rightharpoonup u_\theta \quad \text { weakly in } L^{2^*}(\Omega), \\
        & u_n \rightarrow u_\theta \quad \text { strongly in } L^r(\Omega) \text { for } 2<r<2^*, \\
        & u_n \rightarrow u_\theta \quad \text { a.e. on } \Omega.
    \end{align*}
    Next, we prove that $u_\theta \in S_c^{+}$ is a critical point of $E_\theta$ constrained on $S_c^{+}$. 
    
    Let $w_n=$ $u_n-u_\theta$. Since $\left(\left.E_\theta\right|_{S_c^{+}}\right)^{\prime}\left(u_n\right) \rightarrow 0$, there exists $\lambda_n$ such that $E_\theta^{\prime}\left(u_n\right)-\lambda_n u_n^{+} \rightarrow 0$. Let $\lambda_\theta$ be the Lagrange multiplier corresponding to $u_\theta$. Similar to the proof of Theorem \ref{C}, we have 
    $$
    \int_{\Omega}|\nabla w_n|^2 dx=\theta{\mathbb{D}}(w_n)+o_n(1).
    $$
    Hence, we assume that $\int_{\Omega}\left|\nabla w_n\right|^2 d x \rightarrow \theta l \geq 0,\ {\mathbb{D}}(w_n)\rightarrow l \geq 0$. From the definition of $\mathcal{S}_H$ we deduce that
    $$
    \int_{\Omega}\left|\nabla w_n\right|^2 d x \geq \mathcal{S}_H{\mathbb{D}}(w_n)^{\frac{1}{2_\alpha^*}},
    $$
    implying $\theta l \geq \mathcal{S}_H l^\frac{1}{2_\alpha^*}$. 

    \textbf{Claim:} $l=0$.
    
    Suppose on the contrary that $l>0$, then $ l \geq \mathcal{S}_H^{\ \ \frac{N+\alpha}{\alpha+2} }\theta^{-\frac{N+\alpha}{\alpha+2} }$, implying that
    $$
    E_\theta\left(w_n\right) \geq \frac{\alpha +2}{2N+2\alpha }\mathcal{S}_H^{\frac{N+\alpha }{\alpha+2}}\theta^{\frac{2-N}{\alpha+2}}+o_n(1).
    $$
    Further, using the Brézis-Lieb Lemma (see [Lemma \ref{L}]), one gets
    \begin{align}\label{aq1}
        E_\theta\left(u_n\right)=E_\theta\left(u_\theta\right)+E_\theta\left(w_n\right)+o_n(1) \geq \nu_{c, \theta}+\frac{\alpha +2}{2N+2\alpha }\mathcal{S}_H^{\frac{N+\alpha }{\alpha+2}}\theta^{\frac{2-N}{\alpha+2}}+o_n(1) .
    \end{align}
    On the other hand, using Proposition \ref{D} which will be proved in the next section, \eqref{ao} and Lemma \ref{H}, we have
    $$
    m_\theta=m_1+o_\theta(1)<E\left(u_c\right)+\frac{\alpha +2}{2N+2\alpha }\mathcal{S}_H^{\frac{N+\alpha }{\alpha+2}}+o_\theta(1)=\nu_{c, \theta}+\frac{\alpha +2}{2N+2\alpha }\mathcal{S}_H^{\frac{N+\alpha }{\alpha+2}}\theta^{\frac{2-N}{\alpha+2}}+o_\theta(1) .
    $$
    Taking $\epsilon$ small enough such that $\theta$ close to $1^{-}$ for all $\theta \in(1-\epsilon, 1)$, we obtain
    \begin{align}\label{aq2}
        m_\theta+\xi<\nu_{c, \theta}+\frac{\alpha +2}{2N+2\alpha }\mathcal{S}_H^{\frac{N+\alpha }{\alpha+2}}\theta^{\frac{2-N}{\alpha+2}},
    \end{align}
    for some $\xi>0$ independent of $\theta \in(1-\epsilon, 1)$. Combining \eqref{aq1}, \eqref{aq2} and that $\lim _{n \rightarrow \infty} E_\theta\left(u_n\right)=m_\theta$, we get a self-contradictory inequality
    $$
    \nu_{c, \theta}+\frac{\alpha +2}{2N+2\alpha }\mathcal{S}_H^{\frac{N+\alpha }{\alpha+2}}\theta^{\frac{2-N}{\alpha+2}} \leq m_\theta<m_\theta+\xi<\nu_{c, \theta}+\frac{\alpha +2}{2N+2\alpha }\mathcal{S}_H^{\frac{N+\alpha }{\alpha+2}}\theta^{\frac{2-N}{\alpha+2}}.
    $$
    This completes the claim. 
    
    Hence, $u_n \rightarrow u_\theta$ strongly in $H_0^1(\Omega)$ and $u_\theta$ is a critical point of $E_\theta$ constrained on $S_c^{+}$ at the level $m_\theta$. By Lagrange multiplier principle, $u_\theta$ satisfies
    $$
    -\Delta u_\theta  =\lambda u_\theta+\theta (I_\alpha*|u_\theta|^{2_\alpha^*})|u_\theta|^{2_\alpha^*-2}u_\theta+a\theta(I_\alpha*|u_\theta|^p)|u_\theta|^{p-2}u_\theta,\text { in }\Omega.
    $$
    Multiplying $u_\theta^{-}$ for the equation of $u_\theta$ and integrating on $\Omega$, we have
    $$
    \int_{\Omega}\left|\nabla u_\theta^{-}\right|^2 d x=0
    $$
    implying that $u_\theta \geq 0$. By strong maximum principle, $u_\theta>0$ and thus solves \eqref{aq} when $a\ge0$. 
    
    Using the Pohozaev identity and noticing that $\Omega$ is star-shaped with respect to $0$, we deduce $u_\theta \in \mathcal{G}_\theta$ and complete the proof.
\end{proof}

Now we are prepared to prove the existence of bounded PS sequence.

\begin{proposition}\label{C2}
    Under the assumptions of Theorem \ref{B}, there exists a $H_0^1(\Omega)$-bounded sequence ${u_n} \subset S_c^+$ such that and that $\underset{n\to \infty}{lim} E(u_n) = m(c)$ as $
    n\to \infty$.
\end{proposition}
\begin{proof}
    By Proposition \ref{G}, we can take $\theta_n \rightarrow 1^{-}$and $u_n=u_{\theta_n}$ solving \eqref{aq}. When $a\ge0$. Using Lemma \ref{H} and choosing $n$ large such that $\theta_n$ is close to $1^{-}$ enough, we get $m_{\theta_n} \leq 2 m_1$.
    When $a=0$, noticing $u_{\theta_n} \in \mathcal{G}_{\theta_n}$, 
    \begin{align*}
        2 m_1 \geq m_{\theta_n}=E_{\theta_n}\left(u_n\right)>\left(\frac{1}{2}-\frac{1}{2\cdot 2_\alpha^* \theta_n}\right) \int_{\Omega}\left|\nabla u_n\right|^2 d x   
    \end{align*}
    holds through taking $n$ enough such that $\theta_n$ is close to $1^{-}$. When $a>0$ and $2_\alpha^\sharp<p<2_\alpha^*$, further using the fact that $u_{\theta_n} \in \mathcal{G}_{\theta_n}$, we obtain for $n$ large such that $p \eta_p \theta_n>1$,
    $$
    2 m_1 \geq m_{\theta_n}=E_{\theta_n}\left(u_n\right)>\left(\frac{1}{2}-\frac{1}{2p \eta_p \theta_n}\right) \int_{\Omega}\left|\nabla u_n\right|^2 dx.
    $$
    Hence, $\left\{u_n\right\}$ is $H_0^1(\Omega)$-bounded.  
    
    Since $E_{\theta_n}\left(u_n\right)=m_{\theta_n}$, by using Lemma \ref{H} we get
    $$
    \lim _{n \rightarrow \infty} E\left(u_n\right)=\lim _{n \rightarrow \infty} E_{\theta_n}\left(u_n\right)=\lim _{n \rightarrow \infty} m_{\theta_n}=m_1 .
    $$ 
    Note that $m_1=m(c)$. Moreover, using again that $\left\{u_n\right\}$ is $H_0^1(\Omega)$-bounded, we have
    $$
    \left.E'\right|_{S_c^{+}}\left(u_n\right)=\left.E_{\theta_n}'\right|_{S_c^{+}}\left(u_n\right)+o_n(1) \rightarrow 0 \text { as } n \rightarrow \infty .
    $$
    The proof is complete.
\end{proof}

\begin{proof}[proof of Theorem \ref{E}]
    By Proposition \ref{C2}, there is a $H_0^1(\Omega)$-bounded sequence $\left\{u_n\right\} \subset S_c^{+}$ such that $\lim _{n \rightarrow \infty} E\left(u_n\right)=m(c)$ and that $\left(\left.E\right|_{S_c^{+}}\right)^{\prime}\left(u_n\right) \rightarrow 0$ as $n \rightarrow \infty$. Up to a subsequence, we assume that
    \begin{align*}
        & u_n \rightharpoonup u_\theta \quad \text { weakly in } H_0^1(\Omega), \\
        & u_n \rightharpoonup u_\theta \quad \text { weakly in } L^{2^*}(\Omega), \\
        & u_n \rightarrow u_\theta \quad \text { strongly in } L^r(\Omega) \text { for } 2<r<2^*, \\
        & u_n \rightarrow u_\theta \quad \text { a.e. on } \Omega.
    \end{align*}
    Let $w_n=u_n-\tilde{u}_c$. Since $\left.E'\right|_{S_c^{+}}\left(u_n\right) \rightarrow 0$, there exists $\lambda_n$ such that $E^{\prime}\left(u_n\right)-\lambda_n u_n^{+} \rightarrow 0$. Let $\tilde{\lambda}_c$ be the Lagrange multiplier correspond to $\tilde{u}_c$. Similar to the proof to Theorem \ref{C}, we have
    $$
    \int_{\Omega}|\nabla w_n|^2 dx={\mathbb{D}}(w_n)+o_n(1).
    $$
    Hence, we assume that $\int_{\Omega}\left|\nabla w_n\right|^2 d x \rightarrow l \geq 0,\ {\mathbb{D}}(w_n)\to l\ge 0$. From the definition of $\mathcal{S}_H$ we deduce that
    $$
    \int_{\Omega}\left|\nabla w_n\right|^2 d x \geq \mathcal{S}_H{\mathbb{D}}(w_n)^{\frac{1}{2_\alpha^*}}
    $$
    implying $l \geq \mathcal{S}_H l^\frac{1}{2_\alpha^*}$. 
    
    \textbf{Claim:} $l=0$.

    Suppose on the contrary that $l>0$, then $l \geq \mathcal{S}_H^\frac{N+\alpha}{\alpha+2}$, implying that
    $$
    E\left(w_n\right) \geq \frac{\alpha+2}{2N+2\alpha}\mathcal{S}_H^{\ \ \frac{N+\alpha}{\alpha+2}}+o_n(1) .
    $$
    Further, using the Brézis-Lieb Lemma (Lemma \ref{L}) one gets
    \begin{align}
        E\left(u_n\right)=E\left(\tilde{u}_c\right)+E\left(w_n\right)+o_n(1) \geq \nu_c+\frac{\alpha+2}{2N+2\alpha}\mathcal{S}_H^{\ \ \frac{N+\alpha}{\alpha+2}}+o_n(1),\label{ar}
    \end{align}
    where $\nu_c$ is the energy level definited in Theorem \ref{C}. On the other hand, Proposition \ref{D} yields that
    \begin{align}
        E\left(u_n\right)+o_n(1)=m(c)<\nu_c+\frac{\alpha+2}{2N+2\alpha}\mathcal{S}_H^{\ \ \frac{N+\alpha}{\alpha+2}}.\label{as}
        \end{align}
    Combining \eqref{ar} and \eqref{as} 
    $$
    \nu_c+\frac{\alpha+2}{2N+2\alpha}\mathcal{S}_H^{\ \ \frac{N+\alpha}{\alpha+2}}+o_n(1)<m(c)<\nu_c+\frac{\alpha+2}{2N+2\alpha}\mathcal{S}_H^{\ \ \frac{N+\alpha}{\alpha+2}}
    $$
    holds when $n$ large. This is a contradiction. So $l=0$. Hence, $u_n \rightarrow u_\theta$ strongly in $H_0^1(\Omega)$. Then, quite similar to the proof to Proposition \ref{J} we can complete the proof.
\end{proof}

\section{The estimation of energy level for mountain path solutions} 

In this section, we will estimate the energy level of mountain path solutions, which is key in the  concentration compactness arguement appears in proof of Proposition \ref{J} and Theorem \ref{E}.

\begin{proposition}\label{D}
    Let $N\ge 3$,$a\ge 0$ and $2_\alpha^\sharp<p<2_\alpha^*$. Under the assumptions of Theorem \ref{B}, for $a=0$, we further assume that one of the following holds:
    \begin{itemize}
	    \item[$\bullet$] $N\in \left \{3,4\right \}$,
	    \item[$\bullet$] $N\ge 5$ and $ \alpha \in [N-2,N)$.
    \end{itemize}
    For $a>0$, we further assume that one of the following holds:
    \begin{itemize}
	    \item[$\bullet$] $N\ge 6$, $\alpha \in(0,N-2]$,
	    \item[$\bullet$] $N=5$, $ \alpha \in (0,5)$, and $ \max\left \{\frac{7+2\alpha}{6},2_\alpha^\sharp\right \}<p<2_\alpha^*$,
	    \item[$\bullet$] $N=4$, for $\alpha\in (0,2)$, $\frac{3+\alpha}{2}<p<2_\alpha^*$ and for $\alpha\in [2,4)$, $2_\alpha^\sharp<p<2_\alpha^*$,
	    \item[$\bullet$] $N=3$, for $\alpha\in (0,1)$, $\frac{5+2\alpha}{2}<p<2_\alpha^*$ and for $\alpha\in [1,3)$, $2_\alpha^\sharp<p<2_\alpha^*$.
    \end{itemize}
    Then 
    \begin{align}
        m(c)\le \nu_c+\frac{\alpha+2}{2N+2\alpha}\mathcal{S}_H^\frac{N+\alpha}{\alpha+2},\label{at} 
    \end{align}
    where $\nu_c$ is the energy level of the normalized ground state defined by Theorem \ref{C}.
\end{proposition}

Let  $\xi \in C_{0}^{\infty}(\Omega)$  be the radial function, such that  $\xi(x) \equiv 1$  for  $0 \leq|x| \leq R,\ 0 \leq   \xi(x) \leq 1$  for $R \leq|x| \leq 2 R, \xi(x) \equiv 0$ for $|x| \geq 2 R$, where $B_{2 R} \subset \Omega$. Take $v_{\epsilon}=\xi U_{\epsilon}$  where
$$
U_{\epsilon}=[N(N-2)]^{(N-2) / 4}\left(\frac{\epsilon}{\epsilon^{2}+|x|^{2}}\right)^{(N-2) / 2}.
$$

Before proceeding with the proof of Proposition \ref{D}, we need the following lemmas.

\begin{lemma}(\cite[Lemma 7.1]{MR4476243})\label{K} 
    Let $N\ge3$, and $\omega$ be the area of the unit sphere in $\mathbb{R}^N$. Then it holds that    
    \begin{align*}
        \left \| \nabla v_{\epsilon} \right \|_2^2=\mathcal{S}^{\frac{N}{2}}+O\left(\epsilon^{N-2}\right),\ \ \left \| v_{\epsilon} \right \|_{2^*}^{2^*}=\mathcal{S}^{\frac{N}{2}}+O\left(\epsilon^{N}\right).
    \end{align*}

    \begin{equation} \nonumber
        \left\|v_{\epsilon}\right\|_{q}^{q}=
        \left\{\begin{aligned}
            &K \varepsilon^{N-\frac{(N-2)}{2} q}+o\left(\varepsilon^{N-\frac{(N-2)}{2} q}\right),  \text { if } \frac{N}{N-2}<q< 2^{*}, \\
            &\omega \varepsilon^{\frac{N}{2}}|\log \varepsilon|+O\left(\varepsilon^{\frac{N}{2}}\right),  \text { if } q=\frac{N}{N-2}, \\
            &\omega\left(\int_{0}^{2} \frac{\xi^{q}(r)}{r^{(N-2) q-(N-1)}}\right) \varepsilon^{\frac{q(N-2)}{2}}+o\left(\varepsilon^{\frac{q(N-2)}{2}}\right),  \text { if }  1\le q< \frac{N}{N-2}.
        \end{aligned}
        \right.
    \end{equation}
\end{lemma}

\begin{lemma}(Lemma 3.13 in \cite{MR4741627})\label{M}
If $N\ge3$, $\alpha \in (0,N)$. Then it holds that
\begin{align*}
    {\mathbb{B}_p}(v_\epsilon) \ge O(\epsilon^{2q(1-\eta_q)})
\end{align*}
and
\begin{align*}
    (\mathcal{A}(N,\alpha )\mathcal{C}(N,\alpha))^{\frac{N}{2}}\mathcal{S}_H^{\frac{N+\alpha }{2}}-O(\epsilon^\frac{N+\alpha }{2})\le {\mathbb{D}}(v_\epsilon)\le (\mathcal{A}(N,\alpha )\mathcal{C}(N,\alpha))^{\frac{N}{2}}\mathcal{S}_H^{\frac{N+\alpha }{2}}+O(\epsilon^{N-2}).
\end{align*}
\end{lemma}

\begin{lemma}(Lemma 5.4 in \cite{song2024})\label{N}
    For any $0<\phi \in L_{\text {loc }}^{\infty}(\Omega)$, we have as $\epsilon \rightarrow 0^{+}$,
    $$
    \int_{\Omega} \phi v_{\epsilon} d x \leq 2[N(N-2)]^{(N-2) / 4} \sup _{B_{\delta}} \phi \omega_{N} R^{2} \epsilon^{(N-2) / 2}+o\left(\epsilon^{(N-2) / 2}\right).
    $$
\end{lemma}

\begin{lemma}\label{O}
    For any $0<\phi \in L_{\text {loc }}^{\infty}(\Omega)$, we have as $\epsilon \rightarrow 0^{+}$,
    $$
    \int_{\Omega}\int_{\Omega}\frac{\phi(x)\left |v_{\epsilon}(y)\right | ^{2_\alpha^{*}-1}\left |v_{\epsilon}(y)\right | ^{2_\alpha^{*}}}{\left | x-y \right |^{N-\alpha} }dxdy\ge O\left(\epsilon^{\frac{N-2}{2}}\right).
    $$
\end{lemma}

\begin{proof}
    \begin{align*}
        &\int_{\Omega}\int_{\Omega}\frac{\phi(x)\left |v_{\epsilon}(x)\right | ^{2_\alpha^{*}-1}\left |v_{\epsilon}(y)\right | ^{2_\alpha^{*}}}{\left | x-y \right |^{N-\alpha} }dxdy \\
        =&C\int_{B_{2 R}} \frac{(\frac{\epsilon}{\epsilon^{2}+|x|^{2}})^{\frac{\alpha+2}{2}}\phi(x)(\frac{\epsilon}{\epsilon^{2}+|y|^{2}})^{\frac{N+\alpha}{2}}}{\left | x-y \right |^{N-\alpha} }\xi^{2\cdot 2_\alpha^{*}-1}dxdy\\ 
        \geq&C\epsilon^{-\frac{N+2\alpha+2}{2}}\inf _{B_{\delta}}\phi(x)\int_{B_{R}}\int_{B_{R}} \frac{1}{\left ( 1+\left | \frac{x}{\epsilon } \right | \right )^\frac{\alpha+2}{2}\left | x-y \right |^{N-\alpha}\left ( 1+\left | \frac{y}{\epsilon } \right | \right )^\frac{N+\alpha}{2} }dxdy\\ 
        =&C\epsilon^{-\frac{N+2\alpha+2}{2}}\epsilon^{N+\alpha}\inf _{B_{\delta}}\phi(x)\int_{B_{\frac{R}{\epsilon } }}\int_{B_{\frac{R}{\epsilon } }} \frac{1}{\left ( 1+\left | x\right | \right )^\frac{\alpha+2}{2}\left | x-y \right |^{N-\alpha} \left ( 1+\left | y\right | \right )^\frac{N+\alpha}{2} }dxdy\\
        =&C\epsilon^\frac{N-2}{2}\inf _{B_{\delta}}\phi(x)\int_{B_{\frac{R}{\epsilon } }}\int_{B_{\frac{R}{\epsilon } }} \frac{1}{\left ( 1+\left | x\right | \right )^\frac{\alpha+2}{2}\left | x-y \right |^{N-\alpha} \left ( 1+\left | y\right | \right )^\frac{N+\alpha}{2} }dxdy\\
        \ge& O\left(\epsilon^{\frac{N-2}{2}}\right)\int_{B_R}\int_{B_R} \frac{1}{\left ( 1+\left | x\right | \right )^\frac{\alpha+2}{2}\left | x-y \right |^{N-\alpha} \left ( 1+\left | y\right | \right )^\frac{N+\alpha}{2} }dxdy\ \ (\text{Choose }\epsilon<1 )\\
        =& O\left(\epsilon^{\frac{N-2}{2}}\right).
        \end{align*}
\end{proof}

\begin{proof}[Proof to Proposition \ref{D}]
    Let $w_{\epsilon, s}=u_c+s v_\epsilon,\ s \geq 0$ where $u_c$ is given by Theorem \ref{C}. Then $w_{\epsilon, s}>0$ in $\Omega$. If $v \in H_0^1(\Omega)$, then $v$ can be viewed as a function in $H^1\left(\mathbb{R}^N\right)$ by defining $v(x)=0$ for all $x \notin \Omega$. Define $W_{\epsilon, s}=\beta ^{\frac{N-2}{2}} w_{\epsilon, s}(\beta x)$. By taking $\beta=\frac{\left\|w_{\epsilon, s}\right\|_{L^2(\Omega)}}{\sqrt{c}} \geq 1$ we obtain $W_{\epsilon, s} \in H_0^1(\Omega)$ and $\left\|W_{\epsilon, s}\right\|_{L^2(\Omega)}^2=c$. Let $\overline{W}_k=k^{\frac{N}{2}} W_{\epsilon, \widehat{s}}(k x),\ k \geq 1$ where $\widehat{s}$ will be determined below. Set
    \begin{align*}
        \psi(k)=E\left(\overline{W}_k\right)=\frac{k^2}{2} \int_{\Omega}\left|\nabla W_{\epsilon, \widehat{s}}\right|^2 d x-\frac{k^{2\cdot 2_\alpha^*}}{2\cdot 2_\alpha^*}{\mathbb{D}}(W_{\epsilon, \widehat{s}})-a k^{pN-N-\alpha} {\mathbb{B}_p}(W_{\epsilon, \widehat{s}}).  
    \end{align*}
    Then
    \begin{align*}
        \psi^{\prime}(k)=&k \int_{\Omega}\left|\nabla W_{\epsilon, \widehat{s}}\right|^2 d x-k^{2\cdot 2_\alpha^*-1}
        {\mathbb{D}}(W_{\epsilon, \widehat{s}})-a \left ( pN-2N+\mu \right )k^{pN-2N+\mu-1}  {\mathbb{B}_p}(W_{\epsilon, \widehat{s}}).        
    \end{align*}
By Lemma \ref{K}-\ref{M}, we have
$$
    \int_{\Omega}\left|\nabla W_{\epsilon, \widehat{s}}\right|^2 d x =\int_{\Omega}\left|\nabla u_c\right|^2 d x+\widehat{s}^2\mathcal{S}^{\frac{N}{2}}+o_\epsilon (1),
$$
$$
{\mathbb{D}}(W_{\epsilon, \widehat{s}})\ge {\mathbb{D}}(u_c)+\widehat{s}^{2_\alpha^*\cdot 2}(\mathcal{A}(N,\alpha )\mathcal{C}(N,\alpha))^{\frac{N}{2}}\mathcal{S}_H^{\frac{N+\alpha }{2}}+o_\epsilon(1),
$$
$$
{\mathbb{B}_p}(W_{\epsilon, \widehat{s}}) \ge {\mathbb{B}_p}(u_c)+o_\epsilon (1).
$$
We may choose $\widehat{s}$ large such that $\psi^{\prime}(k)<0$ for all $k>1$. Hence, $E\left(\overline{W}_k\right) \leq E\left(W_{\epsilon, \widehat{s}}\right)$. Furthermore, it not difficult to see that $\psi(k) \rightarrow-\infty$ as $k \rightarrow \infty$.
Define $\gamma(t):=W_{\epsilon, 2 t \widehat{s}}$ for $t \in[0,1 / 2]$ and $\gamma(t):=\bar{W}_{2(t-1 / 2) k_0+1}$ for $t \in[1 / 2,1]$, where $k_0$ is large such that $E\left(\bar{W}_{k_0+1}\right)<E\left(u_c\right)$ and $\bar{W}_{k_0+1} \notin \mathcal{G}$. Then $\gamma \in \Gamma$, where $\Gamma$ is defined by \eqref{ak1}. 

\textbf{Claim:} 
\begin{align}\label{au}
    E\left(W_{\epsilon, s}\right)<E\left(u_c\right)+\frac{\alpha +2}{2N+2\alpha }\mathcal{S}_H^{\frac{N+\alpha }{\alpha+2}}.
\end{align}
Then $\sup _{t \in[0,1]} E(\gamma)<$ $E\left(u_c\right)+\frac{\alpha+2}{2N+2\alpha}\mathcal{S}_H^{\frac{N+\alpha}{\alpha+2} }$, implying that $m(c)<\nu_c+\frac{\alpha+2}{2N+2\alpha}\mathcal{S}_H^{\ \ \frac{N+\alpha}{\alpha+2}}.$

In order to conclude it is sufficient to prove \eqref{au}. By the following fact 
\begin{align*}
    (a+b)^r\ge \begin{cases}a^r+b^r,&\text{ if }r\ge 1,\\a^r+b^r+rab^{r-1},&\text{ if }r\ge 2,\\a^r+b^r+rab^{r-1}+ra^{r-1}b,&\text{ if }r\ge 3,\end{cases}
\end{align*}
we get 
\begin{align}\label{au1}
    {\mathbb{B}_p}(w_{\epsilon, s})\ge \begin{cases}{\mathbb{B}_p}(u_c)+{\mathbb{B}_p}(sv_\epsilon),&\text{ if }p\ge 1,\\{\mathbb{B}_p}(u_c)+{\mathbb{B}_p}(sv_\epsilon)+2p\int_{\Omega}(I_\alpha *|u_c|^p)|u_c|^{p-2}u_c (sv_\epsilon) dx,&\text{ if }p\ge 2,\\{\mathbb{B}_p}(u_c)+{\mathbb{B}_p}(sv_\epsilon)+2p\int_{\Omega}(I_\alpha *|u_c|^p)|u_c|^{p-2}u_c (sv_\epsilon) dx\\+2p\int_{\Omega}(I_\alpha *|sv_\epsilon|^p)|sv_\epsilon|^{p-2}sv_\epsilon u_c dx,&\text{ if }p\ge 3.\end{cases}
\end{align}
    Note that
    \begin{align}\label{av}
        E\left(W_{\epsilon, s}\right)= &E\left(w_{\epsilon, s}\right)+\frac{a}{2p}\left(1-\beta ^{2p\left(\eta_p-1\right)}\right){\mathbb{B}_p}(w_{\epsilon, \widehat{s}}) \nonumber \\
        \le  &E\left(u_c\right)+E\left(sv_\epsilon \right)+s \int_{\Omega} \nabla u_c \nabla v_\epsilon d x-s^{2\cdot2_\alpha^*-1}s\int_{\Omega}(I_\alpha*|v_\epsilon|^{2_\alpha^*})|v_\epsilon|^{2_\alpha^*-2}v_\epsilon u_cdx \nonumber \\
        +&\frac{a}{2p}\left(1-\beta ^{2p\left(\eta_p-1\right)}\right)({\mathbb{B}_p}(u_c)+{\mathbb{B}_p}(v_\epsilon)).    
\end{align}
The inequality holds when $2_\alpha^*\ge 2$, i.e. $\alpha\ge N-4$.
    Since $u_c$ satisfies \eqref{aa}, one gets
    \begin{align}
        \int_{\Omega} \nabla u_c \nabla v_\epsilon d x=&\lambda_c \int_{\Omega} u_c v_\epsilon d x+\int_{\Omega}(I_\alpha*|u_c|^{2_\alpha^*})|u_c|^{2_\alpha^*-2}u_cv_\epsilon dx\nonumber \\
        &+a \int_{\Omega}(I_\alpha*|u_c|^p)|u_c|^{p-2}u_cv_\epsilon dx. \label{aw}
    \end{align}
    
    For $a=0$, using \eqref{au1} and \eqref{aw}, we have 
    \begin{align}\label{ax}
        E\left(W_{\epsilon, s}\right)\le  &E\left(u_c\right)+\frac{s^2}{2} \int_{\Omega}\left|\nabla v_\epsilon\right|^2 d x-\frac{s^{2\cdot 2_\alpha^*}}{2\cdot2_\alpha^*} {\mathbb{D}}(v_\epsilon)+\lambda_c s\int_{\Omega} u_c v_\epsilon d x  \nonumber \\
        &+s\int_{\Omega}(I_\alpha*|u_c|^{2_\alpha^*})|u_c|^{2_\alpha^*-2}u_cv_\epsilon dx-s\int_{\Omega}(I_\alpha*|v_\epsilon|^{2_\alpha^*})|v_\epsilon|^{2_\alpha^*-2}v_\epsilon u_cdx.  
    \end{align}
    It is easy to see, there exists $\epsilon_0>0$ and $0<t_0<t_1<+\infty$ such that \eqref{au} holds. Consider the case $t\in [t_0,t_1]$, There exist $K_1=K_1(N,u_c,\lambda_c)$ and $K_2=K_2(N,u_c)$ such that 
    \begin{align*}
        E\left(W_{\epsilon, s}\right)\le E(u_c)+\frac{\alpha +2}{2N+2\alpha }\mathcal{S}_H^{\frac{N+\alpha }{\alpha+2}}+(K_1R^2-K_2)O(\epsilon^{\frac{N-2}{2}})+o(\epsilon^{\frac{N-2}{2}}).
    \end{align*}
    Choose $R$ small such that $K_1R^2-K_2$ and $\epsilon$ small enough, we get \eqref{au}. 

    Notice that
    $$
    \beta ^2=1+\frac{2s}{c} \int_{\Omega} u_c v_\epsilon d x+\frac{s^2}{c} \int_{\Omega}\left|v_\epsilon\right|^2 d x,
    $$
    it follows from Lemma \ref{K} that
    \begin{align}
        1-\beta ^{2p\left(\eta_p-1\right)}=\frac{2s}{c} p\left(1-\eta_p\right) \int_{\Omega} u_c v_\epsilon d x+\begin{cases}O(\epsilon^2),&\text{ if }N\ge 5,\\O(\epsilon^2|\log{\epsilon} |),&\text{ if }N=4,\\O(\epsilon),&\text{ if }N=3.\end{cases}\label{ay}
    \end{align}
    By Lemma \ref{K}, Lemma \ref{M} and the Hardy-Littlewood-Sobolev's inequality \eqref{ae}, we have
    \begin{align}\label{az}
        \int_\Omega \left(I_{\alpha} *\left|u_c\right|^{p}\right)\left|u_c\right|^{p-2}& u_cv_{\epsilon} d x \lesssim\left\|\left|u_c\right|^{p}\right\|_{\frac{2 N}{N+\alpha}}||\left|u_c\right|^{p-1}v_{\epsilon}||_{\frac{2 N}{N+\alpha}}=\left\|u_c\right\|_{\frac{2 N p}{N+\alpha}}^{p}||\left|u_c\right|^{p-1}v_{\epsilon}||_{\frac{2 N}{N+\alpha}}\nonumber \\ 
        &\lesssim\left\|v_{\epsilon}\right\|_{\frac{2 N}{N+\alpha}}=\begin{cases}O(\epsilon^{\frac{2+\alpha }{2}}), &\text { if } 0<\alpha<(N-4)_{+}, \\O(\epsilon^{\frac{N+\alpha }{4}}|\log{\epsilon}|^{\frac{N+\alpha}{2N}}),&\text { if } \alpha=(N-4)_{+},\\O(\epsilon^{\frac{N-2}{2}}),&\text { if }(N-4)_{+}<\alpha<N.\end{cases}
    \end{align}
    Notice $u_c\in \mathcal{G}$ and hence we arrive at 
    \begin{align}
        \lambda_c c<a(\eta_p-1){\mathbb{B}}_p(u_c).\label{ba}
    \end{align}
    For $a>0$, by using \eqref{ay}-\eqref{ba}, one gets 
    \begin{align*}
        E\left(W_{\epsilon, s}\right)\le &E\left(u_c\right)  +\frac{s^2}{2} \int_{\Omega}\left|\nabla v_\epsilon\right|^2 d x-\frac{s^{2\cdot 2_\alpha^*}}{2\cdot2_\alpha^*} {\mathbb{D}}(v_\epsilon)
        +\lambda_c s \int_{\Omega} u_c v_\epsilon d x\\
        &+s\int_{\Omega}(I_\alpha*|u_c|^{2_\alpha^*})|u_c|^{2_\alpha^*-2}u_cv_\epsilon dx+as\int_{\Omega}(I_\alpha*|u_c|^{p})|u_c|^{p-2}u_cv_\epsilon dx\\
        &+\frac{as}{c}\left(1-\eta_p\right) \int_{\Omega} u v_\epsilon d x{\mathbb{B}_p}(u_c)-\frac{as^{2p}}{2p} {\mathbb{B}_p}(v_\epsilon)\\
        & +\begin{cases}O(\epsilon^2),&\text{ if }N\ge 5,\\O(\epsilon^2|\log{\epsilon} |),&\text{ if }N=4,\\O(\epsilon),&\text{ if }N=3.\end{cases}
    \end{align*}
    It is easy to see that there exists $\epsilon_0>0$ and $0<t_0<t_1<+\infty$ such that \eqref{au} holds. Consider the case $t\in [t_0,t_1]$, There exist $K_3=K_3(N,u_c,\lambda_c)$ such that
\begin{align*}
    E\left(W_{\epsilon, s}\right)\le & E(u_c)+\frac{\alpha +2}{2N+2\alpha }\mathcal{S}_H^{\frac{N+\alpha }{\alpha+2}}+O(\epsilon^{\min \left\{N-2,\frac{N+\alpha}{2}\right\}})-O(\epsilon^{\frac{N-2}{2}})-O(\epsilon^{2p(1-\eta_p)})\\
    +&\begin{cases}O(\epsilon^2),&\text{ if }N\ge 5,\\O(\epsilon^2|\log{\epsilon} |),&\text{ if }N=4,\\O(\epsilon),&\text{ if }N=3,\end{cases}+\begin{cases}O(\epsilon^{\frac{2+\alpha }{2}}), &\text { if } 0<\alpha<(N-4)_{+}, \\O(\epsilon^{\frac{N+\alpha }{4}}|\log{\epsilon}|^{\frac{N+\alpha}{2N}}),&\text { if } \alpha=(N-4)_{+},\\O(\epsilon^{\frac{N-2}{2}}),&\text { if }(N-4)_{+}<\alpha<N.\end{cases}
\end{align*}

Next, we compare the growth of the negative leading term and the positive leading term as $\epsilon \to 0^+$ which depends on $N,\ \alpha,\ p$. \\
$\mathrm{(i)}$ $N\ge 6$.
\begin{itemize}
	\item[$\bullet$] If $0<\alpha \le 2\le N-4$, we need $2p(1-\eta_p)<\frac{\alpha+2}{2}$, which implies $p>\frac{2N-2+\alpha}{2N-4}$. Notice $2_\alpha^\sharp>\frac{2N-2+\alpha}{2N-4}$. Hence, when $0<\alpha \le 2$ and $p> 2_\alpha^\sharp$, the conclusion can be drawn. 		
	\item[$\bullet$] If $2<\alpha\le N-2$, we need $2p(1-\eta_p)<2$, which implies $p>\frac{N+\alpha-2}{N-2}$. Notice $2_\alpha^\sharp>\frac{N+\alpha-2}{N-2}$. Therefore, it is sufficient that $2<\alpha\le N-2$ and $p> 2_\alpha^\sharp$.
	\item[$\bullet$] If $N-2<\alpha\le N$, $2p(1-\eta_p)\ge2$ always holds and our method does not work.
\end{itemize}
$\mathrm{(ii)}$ $N=5$.
\begin{itemize}
	\item[$\bullet$] If $0<\alpha \le 1$, we need $2p(1-\eta_p)<\frac{\alpha+2}{2}$, which implies $p>\frac{8+\alpha}{6}$. Notice $2_\alpha^\sharp>\frac{8+\alpha}{6}$. Hence, when $0<\alpha \le 1$ and $p> 2_\alpha^\sharp$, the conclusion can be drawn. 
			
	\item[$\bullet$] If $1<\alpha\le 5$, we need $2p(1-\eta_p)<\frac{3}{2}$, which implies $p>\frac{7+2\alpha}{6}$. Noticing $\frac{7+2\alpha}{6}\le2_\alpha^\sharp$ implies $1<\alpha\le \frac{7}{4}$. Therefore, it is sufficient $1<\alpha\le \frac{7}{4}$ and $p> 2_\alpha^\sharp$ or $\frac{7}{4}<\alpha \le 5$ and $p>\frac{7+2\alpha}{6}$. 	
\end{itemize}
$\mathrm{(iii)}$ $N=4$.
\begin{itemize}
	\item[$\bullet$] If $0<\alpha <4$, we need $2p(1-\eta_p)<1$, which implies $p>\frac{3+\alpha}{2}$. Notice $2_\alpha^\sharp<\frac{3+\alpha}{2}$. Hence, when $0<\alpha< 4$ and $p>\frac{3+\alpha}{2}$, the conclusion can be drawn. 
\end{itemize}
$\mathrm{(iv)}$ $N=3$.
\begin{itemize}
	\item[$\bullet$] If $0<\alpha <3$, we need $2p(1-\eta_p)<\frac{1}{2}$, which implies $p>\frac{5+2\alpha}{2}$. Notice $2_\alpha^\sharp<\frac{5+2\alpha}{2}$. Hence, when $0<\alpha< 3$ and $p>\frac{5+2\alpha}{2}$, the conclusion can be drawn.
\end{itemize}

For $p>2_\alpha^\sharp\ge 2$, i.e. $\alpha \ge N-2$. Using \eqref{au1}, we find that $as\int_{\Omega}(I_\alpha*|u_c|^{p})|u_c|^{p-2}u_cv_\epsilon dx$ no longer appears. For $N=3$ and $1\le \alpha <3$, the order of the negative term $O(\epsilon^{\frac{1}{2}})$ is lower than that of the positive leading term $O(\epsilon)$. For $N=4$ and $2\le \alpha <4$, the order of the negative term $O(\epsilon)$ is lower than that of the positive leading term $O(\epsilon^2\left |\log\epsilon\right |)$. For $N=5$ and $3\le \alpha <5$, the order of the negative leading term $O(\epsilon^\frac{3}{2})$ is lower than that of the positive leading term $O(\epsilon^2)$. All these cases implies \eqref{au} holds. We complete the proof.
\end{proof}

\bibliographystyle{plain}
\bibliography{ref}

\end{document}